\documentclass[12pt]{amsart} 

\usepackage[margin=2.9cm]{geometry}             
\usepackage{color}
\usepackage{graphicx}
\usepackage{amssymb, enumerate,bm}
\usepackage[matrix,arrow,ps]{xy}
\usepackage{epstopdf, amsmath}
\usepackage{paralist}
\newcommand{\C}{\mathbb C}

\newcommand{\R}{\mathbb R}
\newcommand{\B}{\mathbb B}
\newcommand{\transp}{\,^t}

\newcommand{\aaa}{\mathfrak{(a)}}
\newcommand{\fff}{\mathfrak{(f)}}
\newcommand{\ddd}{(\mathfrak{d})}

\newcommand{\bbb}{\mathfrak{(b)}}
\newcommand{\ttt}{\mathfrak{(t)}}

\newcommand{\vnorm}[1]{{\| #1 \|}}

\DeclareMathOperator{\spanc}{span}
\newtheorem{theo}{Theorem}[section]
\newtheorem{lemma}[theo]{Lemma}

\newtheorem{prop}[theo]{Proposition}

\theoremstyle{remark}
\newtheorem{remark}[theo]{Remark}
\theoremstyle{example}

\theoremstyle{definition}
\newtheorem{defi}[theo]{Definition}
\numberwithin{equation}{section}

\begin{document}

\begin{abstract}
We discuss the links between stationary discs, the defect of analytic discs, and $2$-jet determination of CR automorphisms of generic nondegenerate  real submanifolds of $\C^N$ of 
class $\mathcal{C}^4$.

\end{abstract}

\author{Florian Bertrand  and Francine Meylan}
\title[Nondefective stationary discs and $2$-jet determination
]{Nondefective stationary discs and $2$-jet  determination in higher codimension
}

\subjclass[2010]{}


\maketitle 


\section*{Introduction}
In his important paper \cite{tu0}, Tumanov introduced the notion of defect of an analytic disc  $f$ attached to a generic submanifold $ M \subset \Bbb C^N,$  which was defined equivalently by Baouendi, Rothschild and Tr\' epreau \cite{ba-ro-tr}   as  the dimension of the  real vector space of all holomorphic lifts  $\bm {f}$ in $T^*(\Bbb C^N)$ of $f$ attached to the conormal bundle  $N^*(M)$. 
In particular, Tumanov  proved that the  existence of  {\it nondefective}  analytic discs, that is of defect $0$, attached to $M$ implies the wedge extendability of CR functions. In the present paper, we introduce  a  new notion of nondegeneracy of the Levi map which expresses the existence of a nondefective {\it stationary}  disc attached to the {\it  quadric  model} of $M$. Using a deformation argument developed by Forstneri\v{c} \cite{fo} and Globevnik \cite{gl1}, we produce  a family of stationary discs near that  nondefective disc, which are uniquely determined by their  $1$-jet and which cover an open set  of $M.$ 
As an application of this theory, we deduce a $2$-jet determination for CR automorphisms of our generic submanifold $M.$

\begin{theo}\label{chloe}  
Let $M\subset \C^{N}$ be a $\mathcal{C}^4$ generic  real submanifold.    
We assume that $M$ is \ $\mathfrak{D}$-nondegenerate at $p \in M.$  Then any germ  at $p$ of  CR automorphism  of $M$ of class $\mathcal{C}^3$   is  uniquely determined by its $2$-jet at $p.$     
\end{theo}
We refer to Definition \ref{defect} for the notion of  {\it  $\mathfrak{D}$-nondegenerate} submanifold; this notion is closely related to the existence of a nondefective stationary disc (see Lemma \ref{clio}).
In the previous paper by the authors and Blanc-Centi \cite{be-bl-me},  $2$-jet determination  is obtained under the more restrictive  assumption that $M$ is {\it fully nondegenerate} (see Definition 1.2 in \cite{be-bl-me}). Indeed, while fully nondegeneracy imposes the codimension restriction $d \le n,$ $\mathfrak{D}$-nondegeneracy  requires  $d \le 2n.$
For instance, the quadric 
\begin{equation*}
\begin{cases}
\Re e  w_1=|z_1|^2\\
\Re e  w_2=|z_2|^2\\
\Re e  w_3=z_1\overline{z_2}+\overline{z_1}z_2\\
\end{cases}
\end{equation*}
and its perturbations are $\mathfrak{D}$-nondegenerate but not fully nondegenerate.

Stationary discs were introduced by Lempert \cite{le} in his work on the Kobayashi metric of strongly convex domains and studied further in \cite{hu, pa, tu, su-tu}.   
 Their  use in the finite jet determination problem in the framework of finitely smooth submanifolds has been recently developed by several authors \cite{be-bl, be-de-la,be-bl-me,tu3}.   Finite jet determination problems for {\it real analytic} or  $\mathcal{C}^\infty$ real  submanifolds has attracted  considerable attention.
We know from results of Cartan \cite{eca}, Tanaka \cite{ta}, Chern and Moser  
\cite{ch-mo}, that germs of  CR automorphism of a real-analytic Levi nondegenerate 
hypersurface $M$ in $\C^N$ are uniquely determined by their $2$-jet at $p \in M$.
 We also refer  for instance to the articles of Zaitsev \cite{za}, Baouendi, Ebenfelt and Rothschild \cite{BER1}, Baouendi, Mir and Rothschild \cite{BMR}, Ebenfelt, Lamel and Zaitsev \cite{eb-la-za},  Lamel and Mir, \cite{la-mi}, Juhlin \cite{ju}, Juhlin and Lamel \cite{ju-la}, Mir and Zaitsev \cite{mi-za} in the real analytic framework. We point out the works of Ebenfelt \cite{eb}, Ebenfelt and Lamel \cite{eb-la}, and Kim and Zaitsev \cite{ki-za}, Kolar, Zaitsev and the third author \cite{KMZ1},  in the
$ \mathcal{C}^\infty$ setting.

During the completion of this  work, we received a preprint by Tumanov \cite{tu3}, where $2$-jet determination of CR automorphisms is obtained for  $\mathcal{C}^4$ generic strongly pseudoconvex Levi generating submanifold. His approach is based on the stationary discs method in the jet determination problem \cite{be-bl, be-de-la,be-bl-me}, and although apparently similar, his result and ours are in different settings of generic submanifolds of higher codimension. For instance, the quadric 
\begin{equation*}
\begin{cases}
\Re e  w_1=|z_1|^2-|z_2|^2\\
\Re e  w_2=|z_3|^2\\
\end{cases}
\end{equation*}
is $\mathfrak{D}$-nondegenerate but not strictly pseudoconvex. Nevertheless, the result of Tumanov and ours are a step towards the settlement of the following conjecture:

\vspace{0.3cm}
\noindent {\bf Conjecture.} {\it Consider a $\mathcal{C}^4$ generic  real submanifold $M$ Levi nondegenerate in the sense of Tumanov and admitting a nondefective stationary disc passing through  $p \in M$. Then  any germ  at $p$ of  CR automorphism  of $M$ of class $\mathcal{C}^3$ is  uniquely determined by its $2$-jet $p.$}    
\vspace{0.3cm}

We would like to point out that, recently, the second author \cite{me} gave an example of a quadric in $\C^9$ of finite type with $2$ the only H\"ormander number and Levi nondegenerate in the sense of Tumanov, for which the 
$2$-jet determination for biholomorphisms does not hold. Therefore, even in the real analytic case, the question of an optimal bound for the jet determination for nondegenerate (in the sense of Definition \ref{defnondegbe}) submanifolds remains open.
\vspace{0.3cm}

The paper is organized as follows. In Section $1$ and Section $2$, we introduce the setting of the work and the relevant notions; in particular we connect $\mathfrak{D}$-nondegeneracy with the construction of a nondefective stationary disc (Lemma \ref{clio}). In Section 2, we first recall the main result on the existence of stationary discs (Theorem \ref{theodiscscons}) and we prove that stationary discs near an initial nondefective disc are uniquely determined by their  $1$-jet (Proposition \ref{propjetdiscs}) and cover an open set in the conormal bundle (Proposition \ref{propfill}). This last result is particularly important and differs with the usual results in the theory, where it is proved that centers of such discs cover an open sets  in either $T^*\C^N$ or $C^N$. Finally we prove Theorem \ref{chloe}. The paper also contains an appendix on a more general result on  jet determination for stationary discs.

\section{Preliminaries}

We denote by $\Delta$ the unit disc in $\C$, by $\partial \Delta$ its boundary,  and by $\B\subset \C^N$ the unit ball in $\C^N$.      

\subsection{Nondegenerate generic real submanifolds} 
Let $M \subset \C^{N}$ be a  $\mathcal{C}^{4}$ generic   real submanifold of real codimension $d\ge 1$ through $p$. 
After a local biholomorphic change of coordinates, we may assume that $p=0$ and that $M\subset \C^N=\C^n_z\times\C^d_w$ is given locally by the following system of equations
\begin{equation}\label{eqred}
\begin{cases}
r_1= \Re e  w_1- \transp\bar z A_1 z+ O(3)=0\\
\ \ \ \ \vdots \\
r_d=\Re e  w_d- \transp\bar z A_d z+ O(3)=0
\end{cases}
\end{equation}
where $A_1,\hdots,A_d$ are Hermitian matrices of size $n$. In the remainder O(3), $z$ is of weight one and $\Im m w$ of weight two. We set $r=(r_1,\ldots,r_d)$.

We  recall the following biholomorphic invariant notions  introduced by Beloshapka in \cite{be1} and by the authors and Blanc-Centi in \cite{be-bl-me} that coincide in the hypersurface case and in case $N=4$.  
\begin{defi} [Beloshapka \cite{be1}]\label{defnondegbe}
A $\mathcal{C}^{4}$ generic  real submanifold $M$ of $\C^N$ of codimension $d$ and given by (\ref{eqred}) is  \emph{Levi nondegenerate} at 0 in the sense of Beloshapka if the following two conditions are both satisfied
\begin{center}
\begin{tabular}{cl}
$\aaa$ & $A_1$,...,$A_d$ are linearly independent (equivalently on $\R$ or $\C$),\\
\\
$\bbb$ & $\bigcap_{j=1}^d\mathrm{Ker}A_j=\{0\}$.
\end{tabular}
\end{center}
\end{defi}

\begin{defi} [\cite{be-bl-me}]\label{deffully} 
A  $\mathcal{C}^{4}$ generic real submanifold $M$ of $\C^{N}$ of codimension $d$ given by  (\ref{eqred}) is \emph{fully nondegenerate} at 0 if the following two conditions are  satisfied    
\begin{center}
\begin{tabular}{cl}
$\ttt$ & there exists a real linear combination $A=\sum_{j=1}^d c_jA_j$ that is invertible.\\
\\
$\fff$ & there exists $V \in \C^{n}$ such that if we set $D$ to be the $n \times d$ matrix whose $j^{th}$ column is \\

&  $A_jV$, and $B=\transp \overline D A^{-1}D$, then 
$B$ and $\Re e B$ are invertible. 
\\
\end{tabular}
\end{center}
\end{defi}
We point out that condition $\aaa$  implies  a restriction on the codimension, namely $d \leq n^2$, and that condition $\fff$ implies  $d \leq n.$ 
In the present paper, we introduce a less stringent notion of nondegeneracy by removing the assumption on the invertibility of $B$, which will allow us to work with codimension $d\le 2n.$ 
\begin{defi}\label{defect} 
A  $\mathcal{C}^{4}$ generic real submanifold $M$ of $\C^{N}$ of codimension $d$ given by  (\ref{eqred}) is \emph{$\mathfrak{D}$-nondegenerate} at 0 if the following two conditions are  satisfied    
\begin{center}
\begin{tabular}{cl}
$\ttt$ & there exists a real linear combination $A=\sum_{j=1}^d c_jA_j$ that is invertible.\\
\\
$\ddd$ &there exists $V \in \C^{n}$ such that if we set $D$ to be the $n \times d$ matrix whose $j^{th}$ column is \\

&  $A_jV$, and $B=\transp \overline D A^{-1}D$, then $\Re e B$ is invertible. 

\end{tabular}
\end{center}
\end{defi}
We point out that condition $\ddd$ implies that  ${\rm span}_{\R}\{A_1V,\ldots,A_dV\}$ is of dimension $d$, and the converse holds in the strictly pseudoconvex case, that is, when $A$ is positive definite. Recall that condition $\ttt$ was introduced by Tumanov \cite{tu} and is essential for the construction of stationary disc; we say that {\it $M$ is Levi nondegenerate at $0$ 
 in the sense of Tumanov}  if it satisfies conditon $\ttt$.
 
\subsection{Spaces of functions}
We introduce the spaces of functions we need. 
Let $k \geq 0$ be an integer and let $0< \alpha<1$.
We denote by $\mathcal C^{k,\alpha}=\mathcal C^{k,\alpha}(\partial\Delta,\R)$ the space of real-valued functions  defined on $\partial\Delta$ of class 
$\mathcal{C}^{k,\alpha}$. The space  $\mathcal C^{k,\alpha}$ is endowed with  its usual norm
$$\|f\|_{\mathcal{C}^{k,\alpha}}=\sum_{j=0}^{k}\|f^{(j)}\|_\infty+
\underset{\zeta\not=\eta\in \partial\Delta}{\mathrm{sup}}\frac{\|f^{(k)}(\zeta)-f^{(k)}(\eta)\|}{|\zeta-\eta|^\alpha},$$
where $\|f^{(j)}\|_\infty=\underset{\partial\Delta}{\mathrm{max}}\|f^{(j)}\|$.
We set $\mathcal C_\C^{k,\alpha} = \mathcal C^{k,\alpha} + i\mathcal C^{k,\alpha}$ and we equip this space with the norm
$$\|f\|_{\mathcal{C}_{\C}^{k,\alpha}}=
\|\Re e  f\|_{\mathcal{C}^{k,\alpha}}+\|\Im m f\|_{\mathcal{C}^{k,\alpha}}.$$ 
We denote by $\mathcal A^{k,\alpha}$ the subspace of {\it analytic discs} in $\mathcal C_{\C}^{k,\alpha}$ consisting of functions $f:\overline{\Delta}\rightarrow \C$, holomorphic on $\Delta$ with trace on 
$\partial\Delta$ belonging to $\mathcal C_\C^{k,\alpha}$. 
We now denote by $\mathcal A^{1,\alpha}_{0}$ the subspace of $\mathcal C_{\C}^{1,\alpha}$ of functions of the form $(1-\zeta) f$, with  
$f\in \mathcal A^{1,\alpha}$ and  we equip 
$\mathcal A^{1,\alpha}_{0}$ with the norm 
\begin{equation*}
\|(1-\zeta) f\|_{\mathcal A^{1,\alpha}_{0}}
=\vnorm{ f }_{\mathcal{C}_{\C}^{1,\alpha}}
\end{equation*}
which makes it  a Banach space and isomorphic to $\mathcal A^{1,\alpha}$.  We also denote by 
$\mathcal C_{0}^{1,\alpha}$ the subspace of $\mathcal C^{1,\alpha}$ of functions of the form  $(1-\zeta) v$
 with $v\in \mathcal C_\C^{1,\alpha}$. The space $\mathcal C_{0}^{1,\alpha}$ is equipped with the norm
$$\|(1-\zeta) f\|_{\mathcal C_{0}^{1,\alpha}}=\vnorm{ f }_{\mathcal C_\C^{1,\alpha}}.$$ 
Note that $\mathcal C_{0}^{k,\alpha}$ is a Banach space.

\section{Stationary discs and the defect}

\subsection{Stationary discs}

 Let $M$ be a $\mathcal{C}^{4}$ generic  real submanifold of $\C^N$ of codimension $d$ given by  (\ref{eqred}). An analytic disc $f \in (\mathcal A^{k,\alpha})^{N}$ 
is {\it attached to  $M$} if $f(\partial \Delta) \subset M$. Following Lempert \cite{le} and Tumanov \cite{tu} we define
\begin{defi}
A holomorphic disc $f: \Delta \to \C^N$ continuous up to  $\partial \Delta$ and attached to  $M$ is {\it stationary for $M$} if there 
exists a  holomorphic lift $\bm{f}=(f,\tilde{f})$ of $f$ to the cotangent bundle $T^*\C^{N}$, continuous up to 
 $\partial \Delta$ and such that for all $\zeta \in \partial\Delta,\ \bm{f}(\zeta)\in\mathcal{N}M(\zeta)$
where
\begin{equation*}
\mathcal{N}M(\zeta):=\{(z,w,\tilde{z},\tilde{w}) \in T^*\C^{N} \ | \ (z,w) \in M, (\tilde{z},\tilde{w}) \in 
\zeta N^*_{(z,w)} M\setminus \{0\} \},
\end{equation*}
and where 
$$N^*_{(z,w)} M=\spanc_{\R}\{\partial r_1(z,w), \ldots, \partial r_d(z,w)\}$$ is the conormal fiber at $(z,w)$ of $M$. 
The set of these lifts $\bm{f}=(f,\tilde{f})$, with $f$ non-constant, is denoted by $\mathcal{S}(M)$.
\end{defi}
Note that equivalently, an analytic disc $f \in (\mathcal A^{k,\alpha})^{N}$ attached to  $M$ is stationary for $M$ if there exists $d$ real valued functions $c_1, \ldots, c_d : \partial \Delta \to \R$ such that $\sum_{j=1}^dc_j(\zeta)\partial r_j(0)\neq 0$ for all $\zeta \in \partial \Delta$   and such that the map 
$$\zeta \mapsto \zeta \sum_{j=1}^dc_j(\zeta)\partial r_j(f(\zeta), \overline{f(\zeta)})$$ defined on $\partial \Delta$ extends holomorphically on $\Delta$.

The set of such small discs is invariant under CR automorphisms. Recall the following essential result due to Webster \cite{we} in the hypersurface case and to Tumanov \cite{tu} in higher codimension case.
\begin{prop}[\cite{tu}]\label{propco}
 A generic  real submanifold $M$ given by  (\ref{eqred}) satisfies $\ttt$  if and only the conormal bundle $N^*M$ is totally real at $\left(0,\sum_{j=1}^dc_j\partial r_j(0)\right)$, where the $c_1,\ldots,c_d$ are such that $\sum_{j=1}^d c_jA_j$  is invertible.
 \end{prop}

\subsection{Nondefective stationary discs} 

We recall the notion of  defect of an analytic disc, and more precisely of a stationary disc. Following \cite{ba-ro-tr}, we say that a stationary disc $f$ is {\it defective} if it admits a lift 
$\bm{f}=(f,\tilde{f}): \Delta \to T^*\C^N$ such that  $\displaystyle 1/\zeta.\bm{f}=(f,\tilde{f}/\zeta)$ is holomorphic on $\Delta$. The disc is  {\it nondefective} if it is not defective. We investigate the notion of defective disc in case of a quadric. 
 Consider a quadric submanifold $Q\subset \C^N=\C^n_z\times\C^d_w$ of real codimension $d$ given by
\begin{equation}\label{eqquadric}
\begin{cases}
\rho_1= \Re  e w_1- \transp\bar z A_1 z = 0\\
\ \ \ \ \vdots \\
\rho_d=\Re e w_d - \transp\bar z A_d z = 0 
\end{cases}
\end{equation}
where $A_1,\ldots,A_d$ are hermitian matrices of size $n$. We set $\rho=(\rho_1,\ldots,\rho_d)$. In that context, a stationary disc $f$ for $Q$ is {\it defective} if there exists $c=(c_1,\ldots,c_d) \in \R^d\setminus\{0\}$ such that the map $\zeta \mapsto c\partial_z \rho(f(\zeta))= \sum_{j=1}^d c_j \partial_z \rho_j(f(\zeta))$ defined on $\partial \Delta$ extends holomorphically on $\Delta$. The latter  relates  with Tumanov's equivalent  definition of the defect in \cite{tu0}.  In \cite{be-bl-me}, we considered a special family of lifts $\bm{f}=(h,g,\tilde{h},\tilde{g}) \in \mathcal{S}(Q)$ of the form
 \begin{equation}\label{eqinit}
 \bm{f}=\left((1-\zeta)V,2(1-\zeta)\transp \overline{V}A_1V,\ldots,2(1-\zeta)\transp \overline{V}A_dV,(1-\zeta)\transp \overline{V}A, \frac{\zeta}{2} c\right),
 \end{equation} 
 where $V\in \C^n$ and    $A:=\sum_{j=1}^d c_jA_j.$ 
The following lemma is a crucial observation, according to which, condition $\ddd$ in  Definition \ref{defect} implies the existence of a  nondefective stationary disc of the form  (\ref{eqinit}).
\begin{lemma}\label{clio}
A stationary disc $f=\left((1-\zeta)V,2(1-\zeta)\transp \overline{V}A_1V,\ldots,2(1-\zeta)\transp \overline{V}A_dV\right)$, where $V\in \C^n$,  for $Q$  is nondefective if and only if $A_1V,\ldots,A_dV$ are $\R$-linearly independent. 
\end{lemma} 
\begin{proof}
Let $\bm{f}$ be a stationary disc of the form (\ref{eqinit}) and assume that $\sum_{j=1}^d c_jA_jV=0$ for $c=(c_1,\ldots,c_d) \in \R^d\setminus\{0\}$. The disc 
 $$\bm{f}=\left((1-\zeta)V,2(1-\zeta)\transp \overline{V}A_1V,\ldots,2(1-\zeta)\transp \overline{V}A_dV,0, \frac{\zeta}{2} c\right)$$
is a holomorphic lift of $f$, and thus, $f$ is defective. 

Conversely,  assume that  $f=\left((1-\zeta)V,2(1-\zeta)\transp \overline{V}A_1V,\ldots,2(1-\zeta)\transp \overline{V}A_dV\right)$ is defective. It follows that there exists 
$c=(c_1,\ldots,c_d) \in \R^d\setminus\{0\}$ such that the  map 
$$\sum_{j=1}^d c_j \partial_z \rho_j(f(\zeta))=\left(\overline{\zeta}-1\right)\sum_{j=1}^d c_j \transp \overline{V}A_j=\left(\overline{\zeta}-1\right)\sum_{j=1}^d c_j \transp \overline{(A_jV)}$$ 
extends holomorphically on $\Delta$, which can only occur if  $\sum_{j=1}^d c_j \transp \overline{A_jV}=0$.
\end{proof}

\begin{remark}
Note that in the above, we do not assume that $M$ satsifies $\ttt$. 
\end{remark}

\section{Nondefective stationary discs and jet determination}

We consider a quadric submanifold $Q\subset \C^N$ of real codimension $d$  defined by
$\{\rho=0\},$ where $\rho$ is given by
 (\ref{eqquadric}),  satisfying $\ttt$  and an initial  lift of a stationary disc $\bm{f_0}=(h_0,g_0,\tilde{h_0},\tilde{g_0})$ of the form  (\ref{eqinit}) where $c_1,\ldots,c_d$ are chosen such that the matrix $A=\sum_{j=1}^d c_jA_j$ is invertible.
 We define the affine space  $\mathcal{A} \ni \bm{f_0}  $  to be the subset of $(\mathcal{A}^{1,\alpha})^{2N}$ of discs whose value at $\zeta=1$ is $(0,0,0,c/2),$ that is, of  discs  of the form
$$\left((1-\zeta)h,(1-\zeta)g,(1-\zeta)\tilde{h},(1-\zeta)\tilde{g}+\frac{\zeta}{2} c\right).$$
We  set
$$\mathcal{S}_0(M)=\mathcal{S}(M)\cap \mathcal{A} \ni \bm{f_0}$$ 
Let  
$G(\zeta)$ be the  square matrix of size $2N$ given by
\begin{equation*}
G(\zeta)=\left(\begin{array}{cccccccccccc}
\frac{1}{2}I_d & B(\zeta)& 0\\
 0& G_2(\zeta) &  C(\zeta)\\
 0 & 0 &  -i\zeta I_d\\
\end{array}\right), \zeta \in \partial \Delta,
\end{equation*}
where $I_d$ denotes the identity matrix of size $d$, $G_2(\zeta)$ is the following invertible matrix of size $2n$
\begin{equation*}
G_2(\zeta)=\left(\begin{matrix}
\zeta \transp A& I_n &     \\

i\zeta \transp A & -iI_n       \\

\end{matrix}\right),
\end{equation*}
$B(\zeta)$ is the following  matrix of size $d\times 2n$
\begin{equation*}
B=\left(\begin{matrix}
-(A_1)^1 h_0(\zeta)& \hdots  &-(A_1)^n h_0(\zeta) & 0  & \hdots     &  0 \\
\vdots &  & \vdots &  \vdots & & \vdots \\

-(A_d)^1 h_0(\zeta)& \hdots  &-(A_d)^n h_0(\zeta) & 0  & \hdots     &  0 \\

\end{matrix}\right)=(1-\zeta)B_1,
\end{equation*}
and $C(\zeta)$ is the following  matrix of size $2n \times d$ 
$$C=\left(\begin{matrix}
2 \transp h_0 \overline{(A_1)_1} & \hdots  & 2 \transp h_0 \overline{(A_d)_1} \\
\vdots &  & \vdots &    \\
2 \transp h_0 \overline{(A_1)_n} & \hdots  & 2 \transp h_0 \overline{(A_d)_n} \\
-2i \transp h_0 \overline{(A_1)_1} & \hdots  & -2i \transp h_0 \overline{(A_d)_1} \\
\vdots &  & \vdots &    \\
-2i \transp h_0 \overline{(A_1)_n} & \hdots  & -2i \transp h_0 \overline{(A_d)_n} \\

\end{matrix}\right)$$
where $(A_j)_l$ (resp. $(A_j)^l$) denotes, for $j=1,\dots,d$ and $l=1,\ldots,n$, the $l^{\rm th}$ column (resp. row) of $A_j$.
 The following   theorem on the construction of stationary discs with pointwise constraints was obtained in \cite{be-bl-me}. 
 
  \begin{theo}[\cite{be-bl-me}]\label{theodiscscons}
Let $Q\subset \C^{N}$ be a quadric submanifold of real codimension $d$ given by (\ref{eqquadric}) and satisfying satisfying $\ttt$.  Let 
$\bm{f_0}=(h_0,g_0,\tilde{h_0},\tilde{g_0}) \in \mathcal{S}_0(Q)$ of the form  (\ref{eqinit}). Then there exist  open 
neighborhoods $U$ of $\rho$ in $(\mathcal{C}^4(\B))^{d}$  and $V$ of $0$  in $\R^{2N}$, a real number $\varepsilon>0$ and a $\mathcal{C}^1$ map
$$\mathcal{F}_0:U \times V \to  \mathcal{A} $$
 such that:
\begin{enumerate}[i.]

\item $\mathcal{F}_0(\rho,0)=\bm{f_0}$,

\item for all $r\in U$ the map 
$$\mathcal{F}_0(r,\cdot):V\to \{\bm{f} \in \mathcal{S}_0(\{r=0\})\ \ | \  
\|\bm{f}-\bm{f_0}\|_{\mathcal A^{1,\alpha}_{0}}<\varepsilon\}$$
is one-to-one and onto.
\item  the tangent space $T_{\bm {f_0}}\mathcal S_0(Q)$ of  $\mathcal S_0(Q)$ at $\bm {f_0}$  is the kernel of the linear map $$2\Re e \left[\overline{G(\zeta)}\ \cdot \right] : \left(\mathcal{A}^{1,\alpha}_{0}\right)^{2N}
\to  \left(\mathcal C_{0}^{1,\alpha}\right)^{2N}.$$
\end{enumerate}
\end{theo} 
\begin{remark}
We emphasize that in the previous theorem the model quadric is only supposed to satisfy condition $\ttt$. 
\end{remark}
\subsection{Injectivity of the jet map} 
The method used to prove Theorem \ref{theodiscscons} \cite{be-bl-me} based on the computation of certain integers, namely the partial indices and the Maslov index, implies that discs constructed in  Theorem \ref{theodiscscons} are determined by their $k$-jet at $\zeta=1$ , where $k$ is the greatest partial index; in the present case $k=2$. For seek of completeness, we refer to the appendix for a proof of that claim. It is important to note that this result only requires the quadric model to satisfy $\ttt$. In this section, we show that in case the quadric model also satisfies $\ddd$, $1$-jet determination can be achieved. 

Consider the linear jet map 
$$\mathfrak j_{1}: \left(\mathcal{A}^{1,\alpha}\right)^{2N}  \to \mathbb C^{4N}$$ 
mapping ${\bm f}$ to its $1$-jet at $\zeta=1$, namely 
$$\mathfrak j_{1}({\bm f})=\left( {\bm f}(1), \displaystyle \frac{\partial {\bm f}}{\partial \zeta}(1)\right ).$$
\begin{prop}\label{propjetdiscs}
Let $Q\subset \C^{N}$ be a quadric submanifold of real codimension $d$ given by (\ref{eqquadric}), $\mathfrak{D}$-nondegenerate at $0$.  Consider an initial disc $\bm{f_0}  \in \mathcal{S}_0(Q)$ of the form   (\ref{eqinit})
where $V$ and $c=(c_1,\ldots,c_d)$ are respectively given by $\ddd$ and $\ttt$. 
Then there exist an open
neighborhood $U$ of $\rho$ in $(\mathcal{C}^4(\B))^{d}$ and a positive $\varepsilon$ such that for all $r\in U$ the map  $\mathfrak j_{1}$ is injective on 
$$\{\bm{f} \in \mathcal{S}_0(\{r=0\})\ \ | \ \|\bm{f}-\bm{f_0}\|_{1,\alpha}<\varepsilon\}.$$ In other words, such discs are determined by their $1$-jet at $1$.
\end{prop}
\begin{proof}
By the implicit function theorem, it is enough to prove that the restriction of $\mathfrak j_{1}$ to the tangent space $T_{\bm {f_0}}\mathcal{S}_0(Q)$ of $\mathcal{S}_0(Q)$ at the point $\bm {f_0}$ is injective.  By Theorem \ref{theodiscscons}, we have to show that  the  only element of the kernel of the  map $2\Re e \left[\overline{G(\zeta)}\ \cdot \right] : \left(\mathcal{A}^{1,\alpha}_{0}\right)^{2N}
\to  \left(\mathcal C_{0}^{1,\alpha}\right)^{2N}$  with trivial $1$-jet at $\zeta=1$ is the trivial disc.
 Let $\bm{f}=(1-\zeta)(h,g,\tilde{h},\tilde{g})$ be such an element. According to the proof of Lemma 3.4 in \cite{be-bl-me}, we have 
\begin{eqnarray}\label{super}
\begin{cases}
\tilde{g_j}=a_j-\overline{a_j}\zeta, \ a_j \in \C\\
h=A^{-1}(X+Y\zeta),\\
g(\zeta)= -4\Re e \left(\overline{B_1}A^{-1}X\right)+2\overline{B_1}A^{-1}Y-2\overline{B_1}A^{-1}Y\zeta.\\
\end{cases}
\end{eqnarray}
where for $k=1,\ldots,n$, 
\begin{equation} \label{super1}
X_k=2 \sum_{j=1}^d\transp V\overline{(A_j)_k}\Re e (a_j)+\frac{\tilde{y_k}}{2}+i\frac{y_k}{2} \in \C,
\end{equation} 
$\tilde{y_k}, y_k \in \R$ and   
$$Y_k=-2\sum_{j=1}^d \transp V\overline{(A_j)_k}\overline{a_j} \in \C.$$

Since $\bm{f}$ has a trivial $1$-jet at $1$, we must have 
\begin{eqnarray*}
\begin{cases}
a_j \in \R\\
X=-Y\\
\Re e \left(\overline{B_1}A^{-1}X\right)=0.\\
\end{cases}
\end{eqnarray*}  
It follows that 
\begin{eqnarray*}
\overline{B_1}A^{-1}X&=&-\overline{B_1}A^{-1}Y\\
&=&-2 \underbrace{\left(\begin{matrix}
\overline{(A_1)^1 V}& \hdots  &\overline{(A_1)^n V}  \\
\vdots &  & \vdots  \\
\overline{(A_d)^1 V}& \hdots  &\overline{(A_d)^n V}  \\
\end{matrix}\right)}_{\transp\overline{D}}
A^{-1}
\underbrace{\left(\begin{matrix}
 \transp V\overline{(A_1)_1} & \hdots  & \transp V\overline{(A_d)_1} \\
\vdots &  & \vdots  \\
 \transp V\overline{(A_1)_n} & \hdots  & \transp V\overline{(A_d)_n}  \\
\end{matrix}\right)}_{D}
\left(\begin{matrix}
a_1  \\
\vdots  \\

a_d \\
\end{matrix}\right)  \in i\R^d,
 \end{eqnarray*}
 and so 
 $$\Re e (\transp \overline{D}A^{-1}D)\left(\begin{matrix}
a_1  \\
\vdots  \\

a_d \\
\end{matrix}\right)=0.$$ 
It follows from condition $\ddd$ that 
$a_1=\ldots=a_d=0$. This implies, using \eqref{super} that $g=\tilde{g}=0$ and $h=0$. Equation (\ref{super1}) implies that $y_k=\tilde{y_k}=0$. Finally an  inspection of the proof of Lemma 3.4  in \cite{be-bl-me} shows that  $\tilde{h}=0$.  
\end{proof}

\subsection{Filling an open set in the conormal bundle}\label{secfill} 
In \cite{be-bl-me}, we proved that in case the quadric $Q$ is fully nondegenerate at $0$, the family of centers of stationary discs covers an open set of $\C^N$. This relied essentially on the fact that when $V$ is given by condition $\fff$, the matrix $\transp\overline{D}A^{-1}D$ defined above in the proof of Proposition \ref{propjetdiscs} is invertible, which is no longer the case with condition $\ddd$.  Instead of covering an open set of $\C^N$, we show that  the boundaries of the constructed discs cover an open set in the conormal bundle. Proposition \ref{propjetdiscs} is crucial for this approach.
\begin{prop}\label{propfill}
Let $Q\subset \C^N$ be a quadric submanifold of real codimension $d$ given by (\ref{eqquadric}), $\mathfrak{D}$-nondegenerate at $0$. 
Consider an initial disc $\bm{f_0}  \in \mathcal{S}(Q)$ of the form   (\ref{eqinit})
where $V$ and $c=(c_1,\ldots,c_d)$ are respectively given by $\ddd$ and $\ttt$.
Then there exist an open
neighborhood $U$ of $\rho$ in $(\mathcal{C}^4(\B))^{d}$ and a positive $\varepsilon$ such that for all $r\in U$ the set
$$\{e^{-i\theta}.\bm{f}(e^{i\theta})  \ | \ \bm{f} \in \mathcal{S}_0(\{r=0\})\  \|\bm{f}-\bm{f_0}\|_{1,\alpha}<\varepsilon, \ 0\leq \theta<\varepsilon\}$$
covers an open set in the conormal bundle $N^*(\{r=0\})$.

\end{prop}
The following lemma is a version of Lemma 3, Section 15.5 in \cite{Bo}.
\begin{lemma}\label{bog}
Let   $X_1, \dots, X_s$ be  linearly independent vectors in $\Bbb R^s,$  
 $F: \Bbb R^s \longrightarrow \Bbb R^s$ be  a continuous map,  $S$ be a subset of $\{\sum_{j=1}^{s}t_jX_j\ |  \ t=(t_1, \dots, t_s), \ \ t_j \ge 0\}.$ 
 Suppose that  there exit $ \epsilon>0, \ \epsilon'>0, \ \eta>0$ such that
 \begin{enumerate}[i.]
 \item
$|F(t)- \sum_{j=1}^{s}t_jX_j| \le \eta |t| ,\ \ t=(t_1, \dots, t_s), \ \ t_j \ge 0,\ \ |t| \le \epsilon'.$ 
 \item  For any $ t=(t_1, \dots, t_s), \ \ t_j \ge 0,\ \ | {t}| \le \epsilon'$ such that $t$ has at least one zero component, the line segment between $\sum_{j=1}^{s} {t_j}X_{j}$
  and $F({t})$ does not intersect  $S\cap B_{\epsilon},$ 
   where $B_{\epsilon}$ is the ball centered at $0$ of radius  $\epsilon.$ 
\end{enumerate}
Then the image of  $F$ contains $S\cap B_{\epsilon}.$ 
\end{lemma} 
\begin{proof}The proof of this lemma follows  the proof of Lemma 3,  Section 15.5 in \cite{Bo}. It is based on  properties of the homology groups (of a sphere), using  homotopy and  the fact that the set  $\{\sum_{j=1}^{s}t_jX_j \ | \ t=(t_1, \dots, t_s), \ \ t_j \ge 0\} \cap B_{\epsilon'}$ is contractile to the origin in $\Bbb R^s,$  while the set $\{\sum_{j=1}^{s}t_jX_j, \ t=(t_1, \dots, t_s) \  | \ t_j \ge 0\} \cap \overline{ B_{\epsilon'}}$ restricted to  the   $t$ with at least one zero component 
is homeomorphic to a $(s-1)$-sphere in $\Bbb R^s$ with center $x \in S.$
\end{proof}
\begin{lemma}\label{lemopen}
Let $F: \R^{s}\times \R \to \R^{s}$ be  of the form  $F(x,\theta)=  \theta B(x,\theta),$ where  $B: \R^{s}\times \R \to \R^{s}$ is a continuous function while $B(x,0): \R^{s} \to \R^{s}$ is a $\mathcal{C}^1$ diffeomorphism at $0.$ Then  the image of $F$ contains an open set of $\R^{s}.$
\end{lemma}
\begin{proof}
Let 
$G: \R^{s}\times \R \to \R^{s}\times \R$ be the map defined by $G(x, \theta)=(F(x, \theta),\theta).$ Using its  Taylor expansion, we obtain that for every $\epsilon'> 0,$  there exists $\eta> 0 $ such that 
\begin{equation}\label{cato}
|G(x, \theta)-\theta (B(0),1)-\sum_{i=1}^s \theta x_i  \left(\frac{\partial B}{\partial x_i}(0),0\right)|\le \eta |(x,\theta)|,\ \ |(x,\theta)|\le \epsilon'.
\end{equation}
We choose an open cone  $\Gamma < \Gamma_{(B(0),1)(\frac{\partial B}{ \partial x_1}(0),0)\dots (\frac{\partial B}{\partial x_s}(0),0)},$ where  $\Gamma_{(B(0),1)(\frac{\partial B}{ \partial x_1}(0),0)\dots (\frac{\partial B}{\partial x_s}(0),0)}$ is the cone generated by ${(B(0),1),(\frac{\partial B}{ \partial x_1}(0),0),\dots, (\frac{\partial B}{\partial x_s}(0),0)}.$  Following  the proof of Lemma 3 in \cite{Bo},  there exist $\eta'>0$ and $0   <    \epsilon < \epsilon',$  such that the Euclidean distance from $x\in \Gamma\cap B_{\epsilon}$ to the boundary  of $\Gamma_{(B(0),1),(\frac{\partial B}{ \partial x_1}(0),0)\dots (\frac{\partial B}{\partial x_s}(0),0)}\cap  B_{\epsilon'}$ is greater than $ \eta'|x|.$  Choosing $\epsilon'$ small enough in order to have $\eta < < \eta'$  in \eqref{cato}, the assumptions of Lemma \ref{bog} will be satisfied. Therefore we obtain that the image under $G$ contains  $\Gamma\cap B_{\epsilon},$ and  hence the image under $F$ contains an open set of $\R^{s}.$
\end{proof}
In the following proof,  we will be using the notation 
$$u(\Theta)=(u_1(\theta_1),\ldots,u_{4N}(\theta_{4N}))$$
for a vector valued function $u: \R \to \R^{4N}$ and $\Theta=(\theta_1,\ldots,\theta_{4N}) \in \R^{4N}$.  

\begin{proof}[Proof of Proposition \ref{propfill}]
Consider the map 
$$\Psi_r:\{\bm{f} \in \mathcal{S}_0(\{r=0\})\ | \ \|\bm{f}-\bm{f_0}\|_{1,\alpha}<\varepsilon\} \times [0,\varepsilon) \to N^*(\{r=0\})$$
defined by $$\Psi_r({\bm f},\theta)=e^{-i\theta}.{\bm f}(e^{i\theta})=(f(e^{i\theta}),e^{-i\theta}\tilde{f}(e^{i\theta}))$$
In order to prove the proposition, we will show that the  image of 
$\Psi_{\rho}$ covers an open set by using Lemma \ref{lemopen}. We also choose coordinate in $N^*(Q)$ such that $\bm{f_0}(1)=0$.  Note that we write 
$$\Psi_{\rho}({\bm f},\theta)=\bm{f}(1)+\theta\frac{d}{d\theta} \left(e^{-i\theta}.{\bm f}(e^{i\theta})\right) (\Theta)=\theta \frac{d}{d\theta} \left(e^{-i\theta}.{\bm f}(e^{i\theta})\right)(\Theta) $$
for some $\Theta=(\theta_1,\ldots,\theta_{4N}) \in \R^{4N}$ with $0<\theta_j<\theta$, $j=1\ldots,4N$, and hence is of the form  as in Lemma \ref{lemopen}.
We need to prove that the map defined on  $T_{\bm {f_0}}\mathcal{S}_0(Q)$ 
$${\bm f} \mapsto \frac{\partial \Psi_{\rho}({\bm f},0)}{\partial \theta}=\frac{d}{d\theta} \left(e^{-i\theta}.{\bm f}(e^{i\theta})\right)_{|_{\theta=0}}=(if'(1),i\tilde{f}'(1))$$
is a diffeomorphism onto an open set in $T_{\bm {f_0}(1)} N^*Q$.
According to Proposition \ref{propjetdiscs}, the map ${\bm f} \mapsto {\bm f}'(1)$
is injective and so is ${\bm f} \mapsto \frac{d}{d\theta} \left(e^{-i\theta}.{\bm f}(e^{i\theta})\right)_{|_{\theta=0}}$. It follows from Lemma \ref{lemopen}  that 
the image of $\Psi_{\rho}$ covers an open set in  $N^*(Q)$.

\end{proof}

\subsection{Jet determination of CR automorphisms } 
Let $k$ be a positive integer.   
Let $M\subset \C^N$ be a $\mathcal{C}^4$ generic  real submanifold and let $p \in M$. We denote by $Aut^k(M,p)$ the set of germs at $p$  of CR automorphisms $F$ of $M$ of class $\mathcal{C}^k$; in particular we have $F(p)=p$ and $F(M)\subset M$.    
\begin{theo}\label{theojet}
Let $M\subset \C^N$ be a $\mathcal{C}^4$ generic  real submanifold. Assume that $M$ is   $\mathfrak{D}$-nondegenerate at  $p \in M$. Then elements of $Aut^3(M,p)$ are uniquely determined by their $2$-jet at $p$.     
\end{theo}
We sketch the proof which is nearly identical to the one of Theorem 4.1 in \cite{be-bl-me}. Let $F$ be a CR automorphism class $\mathcal{C}^3$ of $M$, fixing $p=0$. Let $\bm{f_0} \in \mathcal{S}_0(Q)$ be a a lift of a stationary disc for $Q$ of the form 
(\ref{eqinit}), where $Q$ is the model quadric approximating $M$. Let $q \in N^*{M}$ close to    $\bm{f_0}(1)$ and let $\bm{f} \in \mathcal{S}_0(M)$ be a lift of stationary disc for 
$M$ obtained in Theorem \ref{theodiscscons} with  $q=e^{-i\theta}.\bm{f}(e^{i\theta})=(f(e^{i\theta}),e^{-i\theta}\tilde{f}(e^{i\theta}))$. Since $F$ has a trivial $2$-jet at $p=0$, the discs 
$F_{*}\bm{f}(\zeta)=\left(F\circ f(\zeta), \tilde{f}(\zeta)\left(d_{f(\zeta)}F\right)^{-1}\right)$ and $\bm{f}$ have the same $1$-jet at $\zeta=1$ and 
therefore coincide according to Proposition \ref{propjetdiscs}. It follows that the lift of $F$,  still denoted by $F$ satisfies $F(q)=q$. Since such points cover an open set in $N^*M$  
by Proposition \ref{propfill}, the map $F$ is the identity.

\section*{Appendix: Jet determination of stationary discs}

Consider the linear jet map 
$$\mathfrak j_{2}: \left(\mathcal{A}^{2,\alpha}\right)^{2N}  \to \mathbb C^{6N}$$ 
mapping ${\bm f}$ to its $2$-jet at $\zeta=1$, namely 
$$\mathfrak j_{2}({\bm f})=\left( {\bm f}(1), \displaystyle \frac{\partial {\bm f}}{\partial \zeta}(1), \displaystyle \frac{\partial^2 {\bm f}}{\partial \zeta^2}(1)\right )\in \mathbb C^{6N}.$$

Let $M$ be a $\mathcal{C}^{4}$ generic  real submanifold of $\C^N$ of codimension $d$  given by  (\ref{eqred}) and satisfying $\ttt$. Consider a lift of stationary disc ${\bm f}=(f,\tilde{f})$ for $M$ satisfying ${\bm f}(1)=(0,\sum_{j=1}^d c_j\partial r_j(0))$ where $\sum_{j=1}^d c_jA_j$  is invertible. It follows from Proposition \ref{propco} and from Chirka (Theorem 33 in \cite{ch}) that such discs are of class $\mathcal{C}^{2,\alpha}$ for any $0<\alpha<1$ near $\zeta=1$. Therefore the map $\mathfrak j_{2}$ is well defined on the space of lifts of stationary discs.
\begin{prop}
Let $Q\subset \C^{N}$ be a quadric submanifold of real codimension $d$ given by (\ref{eqquadric}), satisfying $\ttt$.  Consider an initial disc $\bm{f_0}  \in \mathcal{S}_0(Q)$ of the form   (\ref{eqinit}) where  $c=(c_1,\ldots,c_d)$ is given by 
$\ttt$. Then there exist an open
neighborhood $U$ of $\rho$ in $(\mathcal{C}^4(\B))^{d}$ and a positive $\varepsilon$ such that for all $r\in U$ the map  $\mathfrak j_{2}$ is injective on $\{\bm{f} \in \mathcal{S}(\{r=0\})\ \ | \ \|\bm{f}-\bm{f_0}\|_{1,\alpha}<\varepsilon\}$; in other words, such discs are determined by their $2$-jet at $1$.
\end{prop}

\begin{proof}
By the implicit function theorem, it is enough to prove that the restriction of $\mathfrak j_{2}$ to the tangent space $T_{\bm {f_0}}\mathcal{S}_0(Q)$ of $\mathcal{S}_0(Q)$ at the point $\bm {f_0}$ is injective.  By Theorem \ref{theodiscscons}, we have to show that  the  only element of the kernel of the  map $2\Re e \left[\overline{G(\zeta)}\ \cdot \right] : \left(\mathcal{A}^{1,\alpha}_{0}\right)^{2N}
\to  \left(\mathcal C_{0}^{1,\alpha}\right)^{2N}$  with trivial $2$-jet at $\zeta=1$ is the trivial disc. 
Let $\bm{f}$ be such an element. We have 
\begin{equation}\label{eqker}
\overline{G(\zeta)}\bm{f}+G(\zeta)\overline{\bm{f}}=0.
\end{equation}
The last $d$ rows of that equation are of the form 
$$i\overline{\zeta} \tilde{g_j}(\zeta)-i\zeta \overline{\tilde{g_j}(\zeta)}=0,  \ j=1,\ldots,n,$$
 from which it follows that 
$$\tilde{g_j}=a_j+b_j\zeta+\overline{a_j}\zeta^2, \ a_j \in \C, \ b_j \in \R.$$
Since $\tilde{g_j}$ has a trivial two jet it follows that $\tilde{g_j}\equiv  0$. Solving the system backward, 
the previous $2n$ rows of Equation (\ref{eqker}) are of the form
 \begin{equation*}
\overline{G_2(\zeta)}\left(\begin{matrix}
h  \\
 \tilde{h}\\
\end{matrix}\right) 
+G_2(\zeta)\left(\begin{matrix}
\overline{h}  \\
 \overline{\tilde{h}}\\
\end{matrix}\right) =0.
\end{equation*}
The fact that $h \equiv \tilde{h}\equiv 0$ is a consequence of the fact that all partial indices of $-\overline{G_2^{-1}}G_2$ are $1$. Indeed, 
following the operations that lead to the introduction of  $G_2^{\flat}$ (see Equation (2.11) in \cite{be-bl-me}), we end up with $n$ systems 
 \begin{equation*}
\overline{R(\zeta)}
\left(\begin{matrix}
h_k  \\
 \tilde{h_k}\\
\end{matrix}\right) 
+R(\zeta)\left(\begin{matrix}
\overline{h_k}  \\
 \overline{\tilde{h_k}}\\
\end{matrix}\right) =0.
\end{equation*}
with the abuse of notation $h=\transp Ah$ - A being invertible -, or equivalently 
\begin{equation*}
\left(\begin{matrix}
h_k  \\
 \tilde{h_k}\\
\end{matrix}\right) 
=-\overline{R(\zeta)^{-1}}R(\zeta)\left(\begin{matrix}
\overline{h_k}  \\
 \overline{\tilde{h_k}}\\
\end{matrix}\right) = -\left(\begin{matrix}
0 & \zeta       \\
\zeta & 0  \\
\end{matrix}\right)\left(\begin{matrix}
\overline{h_k}  \\
 \overline{\tilde{h_k}}\\
\end{matrix}\right) =0.
\end{equation*}
According to  Lemma 5.1 in \cite{gl1} we can write  
\begin{equation*}
-\left(\begin{matrix}
0 & \zeta       \\
\zeta & 0  \\
\end{matrix}\right) = \Theta^{-1} \left(\begin{matrix}
\zeta & 0       \\
0 & \zeta  \\
\end{matrix}\right) \overline \Theta
\end{equation*}
where $\Theta: \overline{\Delta} \to Gl_{2}(\C)$ is a smooth map holomorphic on $\Delta$; more explicitly one can take
for instance 
\begin{equation*}
\Theta= \left(\begin{matrix}
1 & -1       \\
i & i  \\
\end{matrix}\right).
\end{equation*}
 Therefore 
$$\Theta  \left(\begin{matrix}
h_k  \\
\tilde{h_k}\\
\end{matrix}\right) =\left(\begin{matrix}
\zeta & 0       \\
0 & \zeta  \\
\end{matrix}\right) \overline{\Theta \left(\begin{matrix}
h_k  \\
\tilde{h_k}\\
\end{matrix}\right) }.$$
It follows that each of the two components of  $\Theta \left(\begin{matrix}
h_k  \\
\tilde{h_k}\\
\end{matrix}\right)$ is a polynomial of degree one, and so are $h_k$ and $\tilde{h_k}$ (due to the explicit form of $\Theta$).  Since they have a trivial 2-jet at $1$, this implies that  $h_k\equiv \tilde{h_k} \equiv 0$, $k=1,\ldots,n$, and thus $h \equiv \tilde{h} \equiv 0$. 
Finally the first $d$ equations implies directly hat $g \equiv 0$.
\end{proof}


\vskip 1cm
{\small
\noindent Florian Bertrand\\
Department of Mathematics,\\
American University of Beirut, Beirut, Lebanon\\
{\sl E-mail address}: fb31@aub.edu.lb\\

\noindent Francine Meylan \\
Department of Mathematics\\
University of Fribourg, CH 1700 Perolles, Fribourg\\
{\sl E-mail address}: francine.meylan@unifr.ch\\
} 

\end{document}